\documentstyle[amsfonts,amstex]{article}

\newtheorem{ftheor}{Th\'{e}or\`{e}me}[section]
\newtheorem{fprop}{Proposition}[section]
\newtheorem{fcor}{Corollaire}[section]
\newtheorem{theor}{Theorem}[section]
\newtheorem{prop}{Proposition}[section]
\newtheorem{cor}{Corollary}[section]
\newtheorem{proof}{Proof}
\begin{document}

\author{Paul POPESCU et Marcela POPESCU, \and {\small D\'{e}partement de
Mathematiques Appliqu\'{e}es, Universit\'{e} de Craiova, } \and {\small rue
Al.Cuza, No.13, 200585 Craiova, Roumanie} \and Courriel{\small : }{\small %
paul\_p\_popescu(at)yahoo.com, marcelacpopescu(at)yahoo.com}}
\title{Foliated vector bundles and riemannian foliations \\
Fibr\'{e}s vectoriels feuill\'{e}t\'{e}s et feuilletages riemanniens
\thanks{%
The paper is published in C. R. Acad. Sci. Paris, Ser. I 349 (2011) 445--449. The
original publication is available at http://www.elsevier.com/}
}
\date{}
\maketitle

\begin{abstract}
The purpose of this Note is to prove that each of the following conditions
is equivalent to that of the foliation ${\cal F}$ is riemannian: 1) the
lifted foliation ${\cal F}^{r}$ on the bundle of $r$-transverse jets is
riemannian for an $r\geq 1$; 2) the foliation ${\cal F}_{0}^{r}$ on the
slashed ${\cal J}_{0}^{r}$ is riemannian and vertically exact for an $r\geq
1 $; 3) there is a positively admissible transverse lagrangian on ${\cal J}%
_{0}^{r}E$, the $r$-transverse slashed jet bundle of a foliated bundle $%
E\rightarrow M$, for an $r\geq 1$.
\end{abstract}

\renewcommand{\abstractname}{R\'esum\'e}

\begin{abstract}
Le but de cette Note est de d\'{e}montrer que chacune des conditions
suivantes est \'{e}quivalente \`{a} celle qu'un feuilletage ${\cal F}$ soit
riemannien: 1) le feuilletage \'{e}lev\'{e} ${\cal F}^{r}$ sur l'espace de
jets $r$-transverses est riemannien pour un certain $r\geq 1;$ 2) le
feuilletage \'{e}lev\'{e} ${\cal F}_{0}^{r}$ sur l'espace r\'{e}duit des
jets $r$-transverses est riemannien et verticalement exact pour un certain $%
r\geq 1$; 3) il existe un lagrangien positif, admissible et transvers sur $%
{\cal J}_{0}^{r}E$, le fibr\'{e} r\'{e}duit des jets $r$-transverses d'un
fibr\'{e} vectoriel feuillet\'{e} $E\rightarrow M$, pour un certain $r\geq 1$%
.
\end{abstract}

% \noindent {\bf Version fran\c{c}aise abr\'{e}g\'{e}e}

Soit ${\cal F}$ un feuilletage de dimension $k$ sur une vari\'{e}t\'{e} $M$.
Un fibr\'{e} $p:E\rightarrow M$ est {\em feuillet\'{e}} s'il y a un atlas
fibr\'{e} tel que les fonctions structurales sont basiques. Il y a aussi un
feuilletage ${\cal F}_{E}$ sur $E$ qui a la m\^{e}me dimension $k$, tel que $%
p$ restrictionn\'{e} \`{a} chaque feuille $F_{E}$ de ${\cal F}_{E}$ est un
diff\'{e}omrphisme locale sur une feuille $F$ de ${\cal F}$. Dans la Note on
utilise principalement des fibr\'{e}s feuillet\'{e}s qui sont affines ou
vectoriels.

Dans \cite[Definition 1.1]{Ta01} on dit qu'un feuilletage ${\cal F}$ est de 
{\em type fini} s'il existe $r\geq 1$ tel que ${\cal F}^{r}$ est
transversalement parall\'{e}lisable. Si de plus, toutes les feuilles de $%
{\cal F}^{r}$ sont relativement compactes alors on dit que ${\cal F}$ est de 
{\em type fini compact}. Aussi dans \cite[Th\'{e}or\`{e}me 1.2.]{Ta01}
prouve-t-on qu'{\em un feuilletage de type fini compact est riemannien}.
Comme un feuilletage transversalement parall\'{e}lisable est riemannien, le r%
\'{e}sultat de Tarquini est am\'{e}lior\'{e} par le r\'{e}sultat qui suit.

\begin{ftheor}
\label{thm:22 copy(1)} Un feuilletage ${\cal F}^{r}$ est riemannien pour un
certain $r\geq 1$ si est seulement si ${\cal F}$ est un feuilletage
riemannien.
\end{ftheor}

Pour le feuilletage induit ${\cal F}_{0}^{r}$ sur le fibr\'{e} vectoriel r%
\'{e}duit ${\cal J}_{\ast }^{r}={\cal J}^{r}\backslash \{\bar{0}\}$, ce Th%
\'{e}or\`{e}me ne peut donner aucune r\'{e}ponse \`{a} la question suivante: 
{\em quand le feuilletage }$F$ {\em est-il\ riemannien, si }${\cal F}%
_{0}^{r} ${\em \ est riemannien pour un certain }$r\geq 1$?

Soit $p:E\rightarrow M$ un fibr\'{e} vectoriel feuillet\'{e}. Un {\em %
lagrangien positif et admissible }sur $E$ c'est une application continue $%
L:E\rightarrow I\!\!R$ dont la r\'{e}striction au fibr\'{e} r\'{e}duit $%
E_{\ast }=E\backslash \{\bar{0}\}\rightarrow M$ (o\`{u} $\{\bar{0}\}$ est
l'image de la section nulle) est diff\'{e}rentiable et il satisfait aux
conditions suivantes: 1) $L$ est d\'{e}fini positif (c'est-\`{a}-dire que la
forme hessienne verticale est d\'{e}finie positive) et $L(x,y)\geq 0=L(x,0)$%
, $(\forall )x\in M$, $y\in E_{x}=p^{-1}(x)$; 2) $L$ est localement
projetable sur un lagrangien transverse $\bar{L}$; 3) il existe une fonction
basique $\varphi :M\rightarrow (0,\infty )$, tel que $(\forall )$ $x\in M$,
il y a au moin un $y\in E_{x}$ de fa\c{c}on que $L(x,y)=\varphi (x)$. Un 
{\em finslerien} est un lagrangien qui est $2$--homog\`{e}ne; s'il est
positif, alors il est toujours admissible. Le fibr\'{e} vertical $VTE=\ker
p_{\ast }\rightarrow E$ peut \^{e}tre consider\'{e} un sous-fibr\'{e}
vectoriel de $\nu F_{E}\rightarrow E$ par la projection canonique $%
TE\rightarrow \nu F_{E}$, puisque $VTE$ est transverse \`{a} $\tau F_{E}$.
On dit qu'une m\'{e}trique riemannienne invariante $G^{\prime }$ sur $\nu
F_{E}$ est {\em verticalement exacte} si sa restriction $G$ aux sections
verticales transverses c'est justement la forme hessienne verticale d'un
lagrangien positif et admissible $L:E\rightarrow I\!\!R$; dans ces
conditions, on dit aussi que le feuilletage ${\cal F}_{E}$ est {\em %
verticalement exact}. A noter que pour un fibr\'{e} affine $p:E\rightarrow M$%
, la hessienne verticale d'un lagrangien $L:E\rightarrow I\!\!R$ est une
forme bilin\'{e}aire sur les fibres du fibr\'{e} vertical $VTE\rightarrow E$%
, d\'{e}finie par les d\'{e}riv\'{e}es partielles de second ordre de $L$, en
utilisant des coordonn\'{e}es sur fibres (v. \cite{MP09F}, pour d\'{e}tails).

\begin{ftheor}
\label{thm:22-1 copy(1)} Soit ${\cal F}$ un feuilletage sur la vari\'{e}t%
\'{e} $M$ et soit ${\cal F}_{0}^{r}$ le feuilletage relev\'{e} sur le fibr%
\'{e} r\'{e}duit ${\cal J}_{0}^{r}$ des jets d'ordre $r$ transverses du fibr%
\'{e} normal $\nu {\cal F}$. Alors ${\cal F}_{0}^{r}$ est riemannien et
verticalement exact, pour un certain $r\geq 1$, si et seulement si ${\cal F}$
est riemannien.
\end{ftheor}

Pour d\'{e}montrer la condition suffisante, on utilise le r\'{e}sultat
suivant.

\begin{fprop}
\label{pr0 copy(1)}Une m\'{e}trique invariante $g$ sur $\nu F$ donne
canoniquement une m\'{e}trique invariante sur $\nu F^{r}$ qui est
verticalement exacte, pour un certain $r\geq 1$.
\end{fprop}

En particulier, une m\'{e}trique invariante $g$ sur $\nu F$ produit un
lagrangien canonique sur ${\cal J}^{r}$, qui provient de la partie v\'{e}%
rticale de la m\'{e}trique verticalement exacte sur $\nu F^{r}$. On peut se
demander si la r\'{e}ciproque est aussi vraie: {\em est-ce que l'existence
d'un lagrangien sur }$J^{r}${\em \ assure le fait que }$F${\em \ est
riemannen}?

\begin{ftheor}
\label{thm:21 copy(1)} Soit $p:E\rightarrow M$ un fibr\'{e} vectoriel
feuillet\'{e} sur la vari\'{e}t\'{e} feuillet\'{e} $(M,{\cal F})$. Il existe
un lagrangien transverse, positif et admissible sur ${\cal J}^{r}E$, pour un
certain $r\geq 1$, si et seulement si le feuilletage ${\cal F}$ est
riemannien.
\end{ftheor}

Le principal outil technique pour prouver la n\'{e}cessit\'{e} des Th\'{e}or%
\`{e}mes \ref{thm:22-1 copy(1)} et \ref{thm:21 copy(1)} ci-dessus est d'un
int\'{e}r\^{e}t particulier, comme il suit.

\begin{fprop}
\label{prop:23 copy(1)} Soit $p_{1}:E_{1}\rightarrow M$ et $%
p_{2}:E_{2}\rightarrow M$ deux fibr\'{e}s vectoriels feuillet\'{e}s sur la
vari\'{e}t\'{e}e feuil\'{e}t\'{e}e $(M,{\cal F})$ et soit $q_{2}:E_{2\ast
}\rightarrow M$ le fibr\'{e} r\'{e}duit. S'il y a un lagrangien transverse,
positif et admissible $L:E_{2}\rightarrow I\!\!R$ et une m\'{e}trique $b$
sur le fibr\'{e} induit $q_{2}^{\ast }E_{1}\rightarrow E_{2\ast }$, qui este
transverse \`{a} l'\'{e}gard de${\cal F}_{E_{2\ast }}$, alors il y a une m%
\'{e}trique sur $E_{1}$ qui este feuillet\'{e}e \`{a} l'\'{e}gard de ${\cal F%
}$.
\end{fprop}

On peut \'{e}noncer, comme un corollaire, le cas particulier $E_{1}=E_{2}=E$
et $b$ est le hessien d'un lagrangien transverse, positif et admissible $%
L:E\rightarrow I\!\!R$, vue comme une m\'{e}trique sur $p^{\ast }E_{\ast
}\rightarrow E$, o\`{u} $p:E\rightarrow M$ est un fibr\'{e} vectoriel f\'{e}%
uillet\'{e}.

\begin{fcor}
Soit $p:E\rightarrow M$ un fibr\'{e} vectoriel f\'{e}uillet\'{e} sur la vari%
\'{e}t\'{e}e feuil\'{e}t\'{e}e $(M,{\cal F})$. S'il y a un lagrangien
transverse, positif et admissible $L:E\rightarrow I\!\!R$, alors il y a une m%
\'{e}trique feuillet\'{e}e sur $E$.
\end{fcor}

Dans le cas particulier o\`{u} $E=\nu F$ et $L$ est la forme quadratique
d'une m\'{e}trique de Finsler feuillet\'{e}e, on peut obtenir qu'un
feuilletage qui a une m\'{e}trique de Finsler transverse soit un feuilletage
riemannien (le probl\`{e}me est propos\'{e} dans \cite{MiMo} comme un cas sp%
\'{e}cial d'un probl\`{e}me propos\'{e} par E. Ghys dans l'Annexe E du livre 
\cite{Mo}; voir \cite{MiMo, JoWo, MP09F}). Un autre cas int\'{e}ressant est
lorsque $E=\nu ^{\ast }F$, sp\'{e}cialement en ce qui concerne la dualit\'{e}
lagrangien - hamiltonien.

Finalement, il est naturel de consid\'{e}rer la question suivante: {\em %
est-ce qu'on peut \'{e}liminer dans l'hypoth\`{e}se du Th\'{e}or\`{e}me \ref%
{thm:22-1 copy(1)} la condition que le feuilletage }${\cal F}_{0}^{r}${\em \
soit verticalement exact?}

\section{Introduction}

Let $M$ be an $n$-dimensional manifold and ${\cal F}$ be a $k$-dimensional
foliation on $M$. We denote the tangent plane field by $\tau F$ and the
normal bundle $\tau M/\tau F$ by $\nu F$. A bundle is called {\em foliated}
if there is an atlas of local trivializations on $E$ such that all the
components of the structural functions are basic ones. In this case a
canonical foliation ${\cal F}_{E}$ on $E$ is induced, having the same
dimension $k$, such that $p$ restricted to leaves is a local diffeomorphism.
In particular, we consider affine and vector bundles that are foliated.
Given a foliated vector bundle, its tensor bundles are foliated vector
bundles. For example, we can consider the transverse vector bundle of
bilinear forms on the fibers of $E$. If $p:E\rightarrow M$ is a foliated
bundle, then ${\cal J}^{1}E\rightarrow M$ is a foliated bundle of $1$-jets
of foliated sections of $E$; a canonical foliation ${\cal F}_{E}^{1}$ on $%
{\cal J}^{1}E$ can be considered. The elements of ${\cal J}^{1}E$ are
equivalence classes $[s]$ of {\em foliated} local sections $s$ of $E$, where
the equivalence relation is coincidence up to order one. The natural
projection $\pi _{0}^{1}:{\cal J}^{1}E\rightarrow E$ is that of an affine
bundle over $E$ with vector space ${\rm Hom}(\nu F,E)$. Indeed, if $(m,e)$
is an element of $E$, the fiber $(\pi _{0}^{1})^{-1}(m,e)$ can be seen as
the affine space of ($k$-dimensional) subspaces $H$ of $T_{(m,e)}E$ such
that $H\cap \ker p_{\ast }=\{0\}$ and $p_{\ast }H\cap \tau F=\{0\}$. So,
there is a free transitive action of ${\rm Hom}(\nu _{m}F,E_{m})$ on the
fiber $(\pi _{0}^{1})^{-1}(m,e)$. In particular, the tangent space to such a
fiber is canonically isomorphic to ${\rm Hom}(\nu _{m}F,E_{m})$. Analogously
one can consider equivalence classes ${\cal J}^{r}E$ of {\em foliated}
sections of $E$, where the equivalence relation is coincidence up to an
order $r\geq 1$; it carries a foliation ${\cal F}_{E}^{r}$. For $r\geq 1$,
the canonical projection $\pi _{r-1}^{r}:{\cal J}^{r}E\rightarrow {\cal J}%
^{r-1}E$ is also an affine bundle, with the director vector bundle ${\rm Hom}%
((\nu F)^{r},E)$). For $r=0$ one obtain a bundle $\pi _{-1}^{r}:{\cal J}%
^{r}E\rightarrow M$. If $p:E\rightarrow M$ is a foliated {\em vector}
bundle, then $\pi _{-1}^{r}:{\cal J}^{r}E\rightarrow M$ is also a foliated
vector bundle and a natural vector subbundle of ${\cal J}^{1}{\cal J}%
^{r-1}E\rightarrow M$, the first jet bundle of $\pi _{-1}^{r-1}:{\cal J}%
^{r-1}E\rightarrow M$. Details can be found, for example, in \cite{GMS}. The
foliated translation is similarly to the setting used in \cite{Ta01}, where
the foliated vector bundle $\pi :\nu F\rightarrow M$ is considered. In this
case, for sake of simplicity, we denote below ${\cal J}^{r}\nu {\cal F}$ by $%
{\cal J}^{r}$ and the lifted foliation on ${\cal J}^{r}$ by ${\cal F}^{r}$.
According to \cite[Definition 1.1]{Ta01}, a foliation ${\cal F}$ is called
of {\em finite type} if there exists $r\geq 1$ such that ${\cal F}^{r}$ is
transversely parallelizable. If moreover all the leaves of ${\cal F}^{r}$
are relatively compact, then ${\cal F}$ is called a {\em compact finite type
foliation}. In \cite[Theorem 1.2.]{Ta01} it is proved that {\em any compact
finite type foliation is riemannian}. Since a transversely parallelizable
foliation is a Riemannian one, the following result improves the result of
Tarquini.

\begin{theor}
\label{thm:22} The lifted foliation ${\cal F}^{r}$ is riemannian for some $%
r\geq 1$ iff ${\cal F}$ is riemannian.
\end{theor}

Considering the induced foliation ${\cal F}_{0}^{r}$ on the slashed vector
bundle ${\cal J}_{\ast }^{r}={\cal J}^{r}\backslash \{\bar{0}\}$, then
Theorem \ref{thm:22} can not give any answer to the following question: {\em %
when is }${\cal F}${\em \ riemannian if }${\cal F}_{0}^{r}${\em \ is
riemannian for some }$r\geq 1$?

A {\em positively admissible lagrangian} on a foliated vector bundle $%
p:E\rightarrow M$ is a continuous map $L:E\rightarrow I\!\!R$ that is asked
to be differentiable at least when it is restricted to the total space of
the slashed bundle $E_{\ast }=E\backslash \{\bar{0}\}\rightarrow M$, where $%
\{\bar{0}\}$ is the image of the null section, such that the following
conditions hold: 1) $L$ is positively defined (i.e. its vertical hessian is
positively defined) and $L(x,y)\geq 0=L(x,0)$, $(\forall )x\in M$ and $y\in
E_{x}=p^{-1}(x)$; 2) $L$ is locally projectable on a transverse lagrangian $%
\bar{L}$; 3) there is a basic function $\varphi :M\rightarrow (0,\infty )$,
such that for every $x\in M$ there is $y\in E_{x}$ such that $L(x,y)=\varphi
(x)$. If a positively transverse lagrangian $F$ is $2$--homogeneous (i.e. $%
F(x,\lambda y)=\lambda ^{2}F(x,y)$, $(\forall )\lambda >0$), then $F$ is
called a {\em finslerian}; it is also a positively admissible lagrangian,
taking $\varphi \equiv 1$, or any positive constant. We can see the vertical
bundle $VTE=\ker p_{\ast }\rightarrow E$ as a vector subbundle of $\nu
F_{E}\rightarrow E$ by mean of the canonical projection $TE\rightarrow \nu
F_{E}$, since $VTE$ is transverse to $\tau F_{E}$. We say that an invariant
riemannian metric $G^{\prime }$ on $\nu F_{E}$ is {\em vertically exact} if
its restriction to the vertical foliated sections is the transverse vertical
hessian of a positively admissible lagrangian $L:E\rightarrow I\!\!R$; in
this case, we say that the foliation ${\cal F}_{E}$ is {\em vertically exact}%
. Notice that if $p:E\rightarrow M$ is an affine bundle, then the vertical
hessian ${\rm Hess}\ L$ of a lagrangian $L:E\rightarrow I\!\!R$ is a
symmetric bilinear form on the fibers of the vertical bundle $VTE$, given by
the second order derivatives of $L$, using the fiber coordinates (see \cite%
{MP09F, Sh1} for more details using coordinates).

\begin{theor}
\label{thm:22-1} Let ${\cal F}$ be a foliation on a manifold $M$ and ${\cal F%
}_{0}^{r}$ be the lifted foliation on the slashed bundle of $r$-jets of
sections of the normal bundle $\nu {\cal F}$. Then ${\cal F}_{0}^{r}$ is
riemannian and vertically exact for some $r\geq 1$ iff ${\cal F}$ is
riemannian.
\end{theor}

In particular, it follows that any invariant metric $g$ on $\nu F$ gives
rise to a canonical lagrangian on ${\cal J}^{r}$, coming from the vertical
part of the vertically exact invariant riemannian metric on $\nu F^{r}$. So,
it is natural to ask for the converse: does the existence of a lagrangian on 
${\cal J}^{r}$ guaranties that ${\cal F}$ is riemannian?

\begin{theor}
\label{thm:21} Let $p:E\rightarrow M$ be a foliated vector bundle over a
foliated manifold $(M,{\cal F})$. There is a positively admissible
lagrangian on ${\cal J}^{r}E$ for some $r\geq 1$ iff the foliation ${\cal F}$
is riemannian.
\end{theor}

\section{Proof of the main results}

\begin{proof}[Proof of Theorem \protect\ref{thm:22}]
The sufficiency is given below by Proposition \ref{pr0}. We prove the
necessity. By construction, the tangent plane field to ${\cal F}^{r}$ is
sent to $\tau F$ by $(\pi _{-1}^{r-1})_{\ast }$. So, in particular, $(\pi
_{-1}^{r-1})_{\ast }$ induces a surjective map $f:\nu F^{r}\rightarrow \nu F$%
. More precisely, for each $m\in M$ and $(m,\lambda )\in {\cal J}^{r}$, $f$
is surjective from $(\nu F^{r})_{(m,\lambda )}$ to $(\nu F)_{m}$. We know by
assumption there exists a (holonomy) invariant metric $g$ on $\nu F^{r}$.
Let $HF^{r}$ denote the $g$-orthogonal of $\ker f$. Because $\nu F^{r}=\ker
f\oplus HF^{r}$ and $f$ is surjective, we have, for all $(m,\lambda )$ as
above, $(HF^{r})_{(m,\lambda )}\simeq (\nu F)_{m}$. This can be reformulated
as $HF^{r}\simeq (\pi _{-1}^{r-1})^{\ast }\nu F$. Recall that the elements
of $({\cal J}^{r})_{m}$ are equivalence classes of foliated sections of $\nu 
{\cal F}$ defined near $m$. Therefore, for each $m$ one can consider the
equivalence class of the zero section of $\nu F$. We denote by $%
s_{0}:M\rightarrow {\cal J}^{r}$ the corresponding section. We have $\pi
_{-1}^{r-1}\circ s_{0}=Id_{M}$ so that $\nu F=(\pi _{-1}^{r-1}\circ
s_{0})^{\ast }\nu F=s_{0}^{\ast }((\pi _{-1}^{r-1})^{\ast }\nu
F)=s_{0}^{\ast }HF^{r}$. So the metric $g$ restricted to $HF^{r}$ gives a
holonomy invariant metric on $\nu F$.
\end{proof}

Each sufficiency of Theorems \ref{thm:22}, \ref{thm:22-1} and \ref{thm:21}
is implied by the following result.

\begin{prop}
\label{pr0}Any invariant metric $g$ on $\nu F$ gives a canonical vertically
exact invariant riemannian metric on $\nu F^{r}$, for any $r\geq 1$.
\end{prop}

\begin{proof}
We proceed by induction over $r\geq 1$. If $\nabla $ is the Levi-Civita
connection of the invariant metric $g$ on $\nu F$ and $\bar{g}$ is the
induced metric tensor on ${\rm End}(\nu F)={\rm Hom}(\nu F,\nu F)$, then we
can consider the invariant metric $g^{1}([s_{1}],[s_{2}])=g(s_{1},s_{2})+%
\bar{g}(\nabla s_{1},\nabla s_{2})$ and the invariant linear connection $%
D_{X}^{1}[s]=[\nabla _{X}s]$ on the foliated vector bundle ${\cal J}%
^{1}\rightarrow M$. Using the decomposition $\nu F^{1}=V\nu {\cal F}%
^{1}\oplus H\nu {\cal F}^{1}$ given by the linear connection $D^{1}$ and the
isomorphisms $V\nu F\cong p^{\ast }\nu F$, $H\nu F\cong p^{\ast }\nu F$, we
consider the metric $G^{1}=p^{\ast }g\oplus p^{\ast }g$ on $\nu F^{1}$.

Let us assume that a riemannian metric $g^{r}$ and a linear connection $%
D^{r} $ have been constructed on the fibers of the vector bundle ${\cal J}%
^{r}\rightarrow M$, for $r\geq 1$. Let us consider the induced metric tensor 
$\bar{g}^{r}$ on ${\rm Hom}(\nu F,{\cal J}^{r})$. The formulas $\tilde{g}%
^{r}([s_{1}],[s_{2}])=g^{r}(s_{1},s_{2})+\bar{g}^{r}(\nabla s_{1},\nabla
s_{2})$ and $\bar{D}_{X}^{1}[s]=[\nabla _{X}s]$ define an invariant metric
and a linear connection respectively on the vector bundle $J^{1}{\cal J}%
^{r}\rightarrow M$. Now, on the vector subbundle ${\cal J}^{r+1}\subset J^{1}%
{\cal J}^{r}$, we consider the induced metric $g^{r+1}$ and the invariant
linear connection $D_{X}^{r}[s]=p^{\prime }(\bar{D}_{X}^{1}[s])$, where $%
p^{\prime }:J^{1}{\cal J}^{r}\rightarrow {\cal J}^{r+1}$ is the orthogonal
projection. Using the decomposition $\nu F^{r+1}=V\nu {\cal F}^{r+1}\oplus
H\nu {\cal F}^{r+1}$ given by the linear connection $D^{r+1}$ and the
isomorphisms $V\nu F^{r+1}\cong p^{\ast }\nu F^{r+1}$, $H\nu F^{r+1}\cong
p^{\ast }\nu F$, we consider the invariant metric $G^{r+1}=p^{\ast
}g^{r+1}\oplus p^{\ast }g$ on $\nu F^{r+1}$ that is vertically exact.
\end{proof}

The main technical tool to prove the necessity of each Theorems \ref%
{thm:22-1} and \ref{thm:21} has independent interest, as follows.

\begin{prop}
\label{prop:23} Let $p_{1}:E_{1}\rightarrow M$ and $p_{2}:E_{2}\rightarrow M$
be foliated vector bundles over a foliated manifold $(M,{\cal F})$ and $%
q_{2}:E_{2\ast }\rightarrow M$ be the slashed bundle. If there are a
positively admissible lagrangian $L:E_{2}\rightarrow I\!\!R$ and a metric $b$
on the pull back bundle $q_{2}^{\ast }E_{1}\rightarrow E_{2\ast }$, foliated
with respect to ${\cal F}_{E_{2\ast }}$, then there is a foliated metric on $%
E_{1}$, with respect to ${\cal F}$.
\end{prop}

\begin{proof}
For each $(m,e_{2})\in E_{2\ast }$ we have a metric (here seen as a
quadratic form) $b_{(m,e_{2})}:(E_{1})_{(m,e_{2})}\rightarrow {\Bbb R}$. We
want a metric $\bar{b}_{m}:(E_{1})_{m}\rightarrow {\Bbb R}$. The idea is to
integrate the dependency on $e_{2}$, using the fact that metrics form a
convex set in the space of quadratic forms. We set 
\[
B_{m}=\{e_{2}\in (E_{2})_{m}\;;\frac{1}{2}\varphi (m)\leq L(e_{2})\leq
\varphi (m)\}. 
\]%
The assumptions on $L$ guaranty that each $B_{m}$ has finite and non zero
measure with respect to any Lebesgue measure ${\rm Leb}$ on $(E_{2})_{m}$.
Indeed $B_{m}$ has to be proper because it is convex and vanishes at the
origin. So $B_{m}$ is compact and non empty because $\varphi (m)$ is in the
image of $B_{m}$, by assumption. The interior of $B_{m}$ is non-void because
of conditions on $L$. We now set 
\[
\bar{b}_{m}=\frac{1}{{\rm Leb}(B_{m})}\int_{B_{m}}b_{(m,e_{2})}\,d{\rm Leb}%
(e_{2}). 
\]%
Note that there is a unique Lebesgue measure on a real vector space up to
multiplicative constant and this indeterminacy is absorbed when we divide by 
${\rm Leb}(B_{m})$.
\end{proof}

Before using this Proposition to prove Theorems \ref{thm:22-1} and \ref%
{thm:21}, we state as a corollary the case when $E_{1}=E_{2}=E$ and $b$ is
the hessian of a positively admissible lagrangian on $E$, seen as a metric
on $p^{\ast }E_{\ast }\rightarrow E$ for some foliated bundle $%
p:E\rightarrow M$.

\begin{cor}
Let $p:E\rightarrow M$ be a foliated vector bundle over a foliated manifold $%
(M,{\cal F})$. If $L:E\rightarrow I\!\!R$ is a positively admissible
lagrangian, then there is a foliated metric on $E$.
\end{cor}

Specializing further to the case $E=\nu {\cal F}$ and $L$ is a foliated
Finsler metric we get back that any foliation having an invariant transverse
Finsler structure is riemannian (the problem is proposed in \cite{MiMo} and
is a special case of a problem presented by E. Ghys in Appendix E of P.
Molino's book \cite{Mo}; see \cite{MiMo, JoWo, MP09F}). Another interesting
special case is when $E=\nu ^{\ast }F$, specially concerning the duality
lagrangian-hamiltonian. Finally, we return to Theorems \ref{thm:22-1} and %
\ref{thm:21}.

\begin{proof}[Proof of Theorems \protect\ref{thm:22-1} and \protect\ref%
{thm:21} ]
The suffiency for both Theorems follow by Proposition \ref{pr0}. We prove
first the necessity of Theorem \ref{thm:21}. Thanks to Proposition \ref%
{prop:23} with $E_{1}=\nu ^{\ast }{\cal F}$ and $E_{2}={\cal J}^{r}E$, it
suffices to construct a metric on $({\cal \pi }_{-1}^{r})_{0}^{\ast }(\nu
^{\ast }{\cal F})$ (again we won't use anything near the zero section of $%
{\cal J}^{r}E$) which is foliated with respect to ${\cal F}^{r}$. At every $%
[s]\in {\cal J}^{r}E_{(0)}$ we have ${\rm Hess}_{[s]}L$ which is a metric on
the vertical part $\ker (\pi _{0}^{r})_{\ast }$ of the tangent bundle of $%
{\cal J}^{r}E$. This vertical part contains $\ker (\pi _{r-1}^{r})_{\ast }$
since $\pi _{0}^{r}=\pi _{0}^{r-1}\circ \pi _{r-1}^{r}$, where $\pi
_{0}^{0}=p:E\rightarrow M$, ${\cal J}^{0}=E$. The vector bundle $\ker (\pi
_{r-1}^{r})_{\ast }$ is associated with the affine bundle $\pi _{r-1}^{r}:%
{\cal J}^{r}\rightarrow {\cal J}^{r-1}$, thus $\ker (\pi _{r-1}^{r})_{\ast
}\simeq (\nu ^{\ast }{{\cal F})}^{r}\otimes E$. So it makes sense to set,
for any $\lambda \in \nu _{m}^{\ast }{\cal F}$, $b_{(m,[s])}(\lambda
)=\left( {\rm Hess}_{[s]}L\right) (\lambda ^{r}\otimes \pi _{0}^{r}([s]))$,
where $b$ and the vertical hessian are seen as quadratic forms and $\lambda
^{r}=\lambda \otimes \cdots \otimes \lambda $ ($r$ times). Thus the
necessity of Theorem \ref{thm:21} follows. Finally, the necessity of Theorem %
\ref{thm:22-1} follows thanks to Theorem \ref{thm:21} using the lagrangian
on ${\cal J}^{r}$ given by the vertical part of the vertically exact
invariant riemannian metric on $\nu F^{r}$.
\end{proof}

Finally, the following question arises: {\em can we drop in Theorem \ref%
{thm:22-1} the condition that }${\cal F}_{0}^{r}${\em \ be vertically exact?}

{\bf Acknowledgements:} The authors thank the referee for the valuable
improvements given to the final form of the Note. The author P.P. was partially supported 
by a CNCSIS Grant, cod. 536/2008, contr. 695/2009.

\end{document}